\documentclass{amsart}
\usepackage{amsmath,amssymb,amsfonts,mathtools}
\usepackage{enumerate}
\usepackage{graphicx}
\usepackage[dvipsnames]{xcolor}
\usepackage{hyperref}

\usepackage[utf8]{inputenc}
\usepackage[T1]{fontenc}

\numberwithin{equation}{section}
\newtheorem{theorem}{Theorem}[section]
\newtheorem{lemma}[theorem]{Lemma}

\newtheorem{prop}[theorem]{Proposition}

\newtheorem{question}[theorem]{Question}

\newtheorem{cor}[theorem]{Corollary}

\theoremstyle{definition}       {

\newtheorem{example}[theorem]{Example}

\newtheorem{defi}[theorem]{Definition}

}

\newcommand{\N}{\mathbb{N}}

\newcommand{\R}{\mathbb{R}}

\newcommand{\eps}{\varepsilon}

\newcommand{\cah}{\mathcal{H}}

\def\beq{\begin{equation}}
\def\eeq{\end{equation}}
\newcommand{\semmi}[1]{{}}
\newcommand{\norm}[1]{\left\lVert #1 \right\rVert}

\newcommand{\bp}{\begin{proof}}
\newcommand{\ep}{\end{proof}}

\DeclareMathOperator{\diam}{diam}

\DeclareMathOperator{\supp}{supp}

\newcommand{\su}{\subset}

\newcommand{\NN}{\mathbb{N}}

\newcommand{\RR}{\mathbb{R}}

\newcommand{\al}{\alpha}

\renewcommand{\phi}{\varphi}

\newcommand{\mj}[1]{}

\def\rr{{\mathbb R}}
\def\R{{\mathbb R}}

\def\nn{{\mathbb N}}

\def\de{\delta}
\def\Om{\Omega}
\def\om{\omega}

    \newcommand\sig\sigma
    \newcommand\Sig\Sigma

\title{Signed null sequences and Hausdorff dimension}
\author{Rich\'ard Balka, Korn\'elia H\'era, and Gergely Kiss}

\address{HUN-REN Alfr\'ed R\'enyi Institute of Mathematics, Re\'altanoda u.~13--15, H-1053 Budapest, Hungary, AND Institute of Mathematics and Informatics, Eszterh\'azy K\'aroly Catholic University, Le\'anyka u.~4, H-3300 Eger, Hungary}
\email{balkaricsi@gmail.com}

\address
{HUN-REN Alfr\'ed R\'enyi Institute of Mathematics, Re\'altanoda u.~13--15, H-1053 Budapest, Hungary, and Mathematical Institute, University of Bonn, Endenicher Allee 60, 53115 Bonn, Germany}
\email{herakornelia@gmail.com}

\address{HUN-REN Alfr\'ed R\'enyi Institute of Mathematics, Re\'altanoda u.~13--15, H-1053 Budapest, Hungary, and  Corvinus University of Budapest, Department of Mathematics,
Fővám tér 13--15, H-1093 Budapest, Hungary}
\email{kigergo57@gmail.com}

\thanks{RB, KH, GK were supported by the National Research, Development and Innovation Office -- NKFIH, grant no.~146922.
RB and GK were also supported by the J\'anos Bolyai Research Scholarship of the Hungarian Academy of Sciences. GK was also supported by NKFIH grants no. FK 142993 and STARTING no. 150576.
KH also received support from the Deutsche Forschungsgemeinschaft (DFG, German Research Foundation) under Germany's Excellence Strategy - GZ 2047/1, Projekt-ID 390685813  as well as SFB 1060}



\begin{document}

\begin{abstract}
We investigate the convergence of signed null sequences of the form
\[
\sum_{n=1}^\infty \varepsilon_n a_n, \quad \varepsilon_n \in \{-1,1\},
\]
where $(a_n)$ tends to zero in $\mathbb{R}^d$.
Our main result shows that for any such sequence, the set of sign sequences yielding convergence has full Hausdorff dimension in the natural ultrametric topology.
This answers a question of Mattila in the one-dimensional case, for which we provide an elementary proof.
Moreover, if $(a_n)\notin \ell^1$ in one dimension, then for every $L\in\mathbb{R}$ the set of sign sequences with sum $L$ also has Hausdorff dimension $1$.
In higher dimensions the analogous statement does not hold in full generality, but it is guaranteed if the sequence has $d$ linearly independent L\'evy vectors.
\end{abstract}

\maketitle

\section{Introduction}

We study the convergence properties of infinite series of the form
\[
\sum_{n=1}^\infty \varepsilon_n a_n, \quad \varepsilon_n \in \{-1, 1\},
\]
where $(a_n)$ is a real or vector-valued sequence converging to zero. Such sign-alternated series appear naturally in real analysis, probability theory, harmonic analysis, and more recently, in fractal geometry.

A fundamental question is: for which sequences $(a_n)$ and for which sign sequences $(\varepsilon_n)$ does the series converge?

The classical starting point is Riemann's rearrangement theorem \cite{riemann}, which asserts that any conditionally convergent real series can be rearranged to converge to any real number or to diverge. This highlights how fragile convergence can be in the absence of absolute summability.

From a probabilistic point of view, Kolmogorov’s three-series theorem \cite{kolmogorov} provides a more stable perspective. When applied to sign-alternated series, it characterizes almost sure convergence in terms of square summability: the series converges for almost every choice of signs precisely when $(a_n) \in \ell^2$.

Beyond mere convergence, the distributional behavior of random sign series has attracted considerable attention.
Answering a question of Erdős \cite{erdos1939}, Solomyak \cite{solomyak1995} showed that if $a_n=\lambda^n$ then the distribution of the random sum $\sum \varepsilon_n a_n$ absolutely continuous for Lebesgue-almost every $\lambda \in (1/2,1)$ with an $L^2$ density. For a summary of related results the reader might consult \cite{peres-schlag-solomyak2000}.

From a geometric measure theory standpoint, an intriguing question emerges: how large is the set of sign sequences $(\varepsilon_n) \in \{-1,1\}^\mathbb{N}$ for which the series converges? Surprisingly, even when convergence is neither absolute nor almost sure, the set of convergent sign sequences can be large in a fractal sense. Specifically, for any null sequence $(a_n)$ in $\mathbb{R}^d$, the set of sign sequences for which the series converges has full Hausdorff dimension (equal to $\dim \Omega = 1$ in the natural ultrametric topology). This answers a question posed by Pertti Mattila \cite{Ma2} in the case $d = 1$. Similar problems were tackled in \cite{CDMV}, and they might be able to solve the problem with sophisticated martingale methods in case of $d=1$. However, our method is much more elementary, and can be generalized to higher dimensions.

More precisely, we fix a null sequence $(a_n) \subset \mathbb{R}^d$, and consider the signed sums $\sum \varepsilon_n a_n$ for $\varepsilon = (\varepsilon_n) \in \Omega := \{-1,1\}^{\mathbb{N}}$. We equip $\Omega$ with the ultrametric
\begin{equation}\label{eqrho}
\rho(\varepsilon, \omega) := 2^{-\min\{n : \varepsilon_n \neq \omega_n\}},
\end{equation}
which makes $(\Omega, \rho)$ a Polish space with Hausdorff dimension 1.

\begin{theorem}
Let $(a_n)$ be a real sequence such that $a_n\to 0$. Then  
 \begin{equation*}
\dim\left\{\eps \in \Omega: \sum_{n=1}^\infty  \varepsilon_n a_n \text{ is convergent} \right\}=1.
\end{equation*} 
\end{theorem}

We next study the so-called achievement set, defined as
\begin{equation*}
S(a_n) := \left\{ x \in \mathbb{R} : \exists \varepsilon \in \Omega \text{ such that } \sum \varepsilon_n a_n = x \right\}.
\end{equation*}
When $(a_n) \in \ell^1$, the set $S(a_n)$ is compact, and can be a Cantor set, a closed interval, or a mixture of both, depending on the decay rate. For instance, if $a_n = (1/2)^n$, then $S(a_n) = [-2, 2]$, whereas if $a_n = (1/3)^n$, the achievement set is a classic Cantor set. For sequences not in $\ell^1$, we have the following result. 

\begin{theorem}
Let $(a_n)$ be a null sequence in $\mathbb{R}$. If $(a_n) \notin \ell^1$, then for every $L \in \mathbb{R}$,
\begin{equation*}
\dim\left\{\varepsilon \in \Omega : \sum_{n=1}^\infty \varepsilon_n a_n = L\right\} = 1.
\end{equation*}
\end{theorem}

For higher dimensions ($d \geq 2$), we show that the set of convergent sequences still has Hausdorff dimension 1. 

\begin{theorem}
Let $d\geq 2$. For any sequence $(a_n)$ in $\RR^d$ with $\|a_n\| \to 0$ we have
\begin{equation*} 
\dim \left\{ \eps \in \Om: \sum_{n=1}^\infty \eps_n a_n  \text{ is convergent} \right\}=1.
\end{equation*}
\end{theorem}

However, the structure of the achievement set $S(a_n) \subset \mathbb{R}^d$ becomes more subtle. Even if $(a_n) \notin \ell^1$, it may happen that $S(a_n) \ne \mathbb{R}^d$. Under certain conditions, such as the lack of a direction $e \in \mathbb{S}^{d-1}$ with $\langle a_n, e \rangle \in \ell^1$, the set $S(a_n)$ is dense in $\mathbb{R}^d$. These considerations are closely related to the classical Lévy–Steinitz theorem on conditionally convergent vector-valued series.

Beyond the classical Lévy–Steinitz theorem, several recent contributions on achievement sets \cite{BGM,GM,filo} have investigated subsums of conditionally convergent series. In contrast to the $\{\pm 1\}$-valued sign sequences considered here, achievement sets are generated by $\{0,1\}$-valued choices, corresponding to the collection of all subsums of a series. Bartoszewicz, G\l \k{a}b and Marchwicki \cite{BGM} showed that such sets may form intervals, Cantor-type sets, or combinations thereof, depending on the decay of the sequence. G\l \k{a}b and Marchwicki \cite{GM} analyzed the planar case in relation to the Lévy–Steinitz theorem, while Marchwicki and Vlasák \cite{filo} studied subsums in finite dimensions and gave conditions under which the achievement set coincides with the full space. These works apply combinatorial and topological methods rather than the Lévy–vector techniques that play a central role in our approach.

In our setting we show that the condition of the Lévy–Steinitz theorem guarantees that the achievement set in dimension $d\ge 2$ is dense, and by adopting the construction of \cite{BGM}, we demonstrate that it is not necessarily closed (i.e., not the whole space). Motivated by these perspectives, we develop the Lévy–vector framework in detail and prove the following result.

\begin{theorem}
Let $(a_n)$ be a null sequence in $\R^d$ with $d$ linearly independent L\'evy vectors.
Then $S(a_n) = \R^d$.
\end{theorem}

This allows us to formulate an analogue of the result for prescribed sums in dimension $d \geq 2$.

\begin{theorem}
Let $(a_n)$ be a sequence in $\RR^d$ with $d$ linearly independent Lévy vectors and $\|a_n\|\to 0$. Then for all $L \in \RR^d$ we obtain
\begin{equation}	
\dim\left\{\eps \in \Om: \sum_{n=1}^\infty  \varepsilon_n a_n =L \right\}=1.
\end{equation}
\end{theorem}

The paper is organized as follows. 
In Subsection~\ref{sec:notation}, we collect the necessary notation and preliminary facts on Hausdorff dimension and ultrametric spaces. 
Section~\ref{s2} contains our main results in the one-dimensional setting, including some preliminary remarks on achievement sets and the proof of their full Hausdorff dimension (Theorem \ref{thm1dim}). 
In Section~\ref{s3}, we turn to higher dimensions and establish the convergence properties of signed null sequences in $\mathbb{R}^d$ (Theorem \ref{thm1hdim}).  
In Section~\ref{s:sums} we investigate the achievement sets of a null sequence in higher dimensions. We show that even if the sequence satisfies natural conditions, the achievement set is not necessarily the whole space $\RR^d$, but it is a dense subset of it if it satisfies the so-called Lévy–Steinitz condition (see condition \eqref{LS}, Proposition \ref{propcex1} and Example \ref{propcex2}). We develop the Lévy–vector framework and provide sufficient conditions ensuring that the achievement set coincides with the full space (see Theorem \ref{thmlevy}). 
At the end of this section Theorem \ref{thmdhdim} shows that if $(a_n)\subset\mathbb{R}^d$ admits a full system of linearly independent Lévy vectors, then for every $L\in\mathbb{R}^d$, the set of sign sequences realizing $L$ has Hausdorff dimension~1.
Finally, Section~\ref{sec:open} lists several open questions and directions for future research.

\subsection{Notation and preliminaries}\label{sec:notation}

Let $(X,\rho)$ be a metric space. The open ball of center $x$ and radius $r$ will be denoted by $B(x,r)$.
For a set $U \su X$, $U_{\de}=\bigcup_{x \in U} B(x,\de)$ denotes the open $\de$-neighborhood of $U$, and $\diam(U)$ denotes the diameter of $U$ that is $\diam(U):=\sup\{\rho(x,y)| x,y\in U\}$.

The metric space $(X,\rho)$ is called an {\it ultrametric space} if the triangle inequality is replaced with the stronger inequality $\rho(x,y) \leq \max(\rho(x,z), \rho(y,z))$ for all $x, y, z \in X$.

Let $s \geq  0$, $\de \in (0,\infty]$ and $A \su X$.
The {\it $s$-dimensional Hausdorff $\de$-premeasure} of $A$ is
    $$
            \cah^s_{\de}(A)
        =
            \inf
            \left\{
                \sum_{i=1}^{\infty}
                    (\diam(U_i))^s
            :
                A \su \bigcup_{i=1}^{\infty} U_i,
            \ %
                \diam(U_i)
            \leq
                \de
            \ %
            (i=1,2,\dots)
            \right\},
    $$
and the {\it $s$-dimensional Hausdorff measure} of $A$ is defined as $\cah^s(A)=\lim_{\de \to 0+} \cah^s_{\de}(A)$.
The {\it Hausdorff dimension} of $A$ is defined as
$\dim A=\sup\{s: \cah^s(X)>0\}$.

For the well-known properties of Hausdorff measures and dimension, see e.\,g.~\cite{Fa} or \cite{Ma}.

In the arguments of this paper we use frequently the following theorem that is called Mass Distribution Principle (\cite[p.60]{Fa} and \cite[Theorem~4.19]{MP}). 

\begin{theorem}[Mass Distribution Principle]\label{thm_MDP}
Let $\mu$ be a mass distribution on a Borel set $F$ and suppose that for some $s\in \mathbb{R}^+$ there are numbers
$c > 0$ and $\delta > 0$ such that
$$\mu(U) \le c(\diam(U)^s)$$ 
for all sets $U$ with $\diam(U)\le \delta$. Then $H^s(F)\ge \mu(F)/c$ and
$s\le \dim F$.
\end{theorem}

For $1 \leq p < \infty$, $\ell^p$ denotes the space of real valued sequences $(a_n)_{n \in \nn}$ with
$\sum_{n=1}^\infty |a_n|^p < \infty$.

\section{Main results in dimension 1}\label{s2}

Recall that $\Omega=\{-1,1\}^{\NN}$ denotes the set of infinite sequences consisting of $\pm 1$. The metric $\rho$ defined as in \eqref{eqrho} provides that $(\Omega,\rho)$ is an ultrametric Polish space and as we noted $\dim \Omega=1$.
For $\eps\in \Omega$ and $n\in \NN$ let $\eps \restriction n$ denote $\eps$ restricted to its first $n$ coordinates. For a finite word $\tau\in \{-1,1\}^n$ let $[\tau]=\{\eps\in \Omega: \eps\restriction n=\tau\}$.

Let $(a_n)$ be a real sequence such that $a_n\to 0$. The following theorem states that the sign sequences $\eps\in \Omega$ for which $\sum_n \eps_n a_n$ converges form a set of maximal Hausdorff dimension in $\Omega$. Moreover, if $(a_n)\not\in \ell^{1}$, then we can even prescribe the limit $\sum_n \eps_n a_n$. Recall that if the sequence $(a_n)$ is in $\ell^{1}$, then for all $\eps\in \Omega$ the sum $\sum_n \eps_n a_n$ converges, and the achievement set $S(a_n)$ is compact. Indeed, $\Omega$ is a compact set and every sum of the form $\sum_n\eps_n a_n$ converges. As the function $\eps \mapsto \sum_n\eps_n a_n$ is continuous as $(a_n)\in \ell^1$, the achievement set is a continuous image of a compact space, hence it is also compact.
\begin{theorem}\label{thm1dim}
Let $(a_n)$ be a real sequence such that $a_n\to 0$. Then
 \begin{equation*}
\dim\left\{\eps \in \Omega: \sum_{n=1}^\infty  \varepsilon_n a_n \text{ is convergent} \right\}=1.
\end{equation*}
If $(a_n) \not\in\ell^{1}$, then for all $L \in \RR$ we obtain
\begin{equation}\label{eq:l1}	
\dim\left\{\eps \in \Om: \sum_{n=1}^\infty  \varepsilon_n a_n =L \right\}=1.
\end{equation}
\end{theorem}
	
\begin{proof}
We may assume that $(a_n) \not \in\ell^{1}$, otherwise every sum is convergent as we noted above. Now we prove \eqref{eq:l1}.
Fix an arbitrary $L\in \RR$ and an integer $k \geq 2$, it is enough to show that
\begin{equation*}
\dim\left\{\eps \in \Om: \sum_{n=1}^\infty  \varepsilon_n a_n =L \right\}\geq \frac{k-2}{k},
\end{equation*}
as we obtain the statement by taking the limit $k\to \infty$.

For $\eps\in \Omega$ and an integer $j\geq 0$ we define
\begin{equation*}
s_j=s_j(\eps)=\sum_{n=1}^{jk} \eps_n a_n,
\end{equation*}
and for all integers $j\geq 0$ let
\begin{equation*} I_j=\{kj+1, \dots,k(j+1)\} \quad
\textrm{and} \quad  \Omega_j=\{-1,1\}^{I_j},
\end{equation*}
and fix $m(j)\in I_j$ such that $|a_{m(j)}|=\max\{|a_n|\colon n\in I_j\}$.
Define $\Lambda\subset \Omega$ as
\begin{equation*}
\Lambda=\{\eps\in \Omega: (s_j-L)(s_{j+1}-s_j)\leq 0 \text{ and }  |s_{j+1}-s_j|\ge |a_{m(j)}| \text{ for all } j\geq 0 \}.
\end{equation*}
It is clear that we can produce some $\eps\in \Lambda\subset\Omega$, hence $\Lambda$ is nonempty.

First we prove that all $\eps\in \Lambda$ satisfies $\sum_n \eps_n a_n=L$. The definition of $\Lambda$ implies that for all $j\geq 0$ we have one of the following cases:
\begin{enumerate}[(i)]
\item $s_j - L$ and $s_{j+1} - L$ have opposite signs, that is,
\begin{equation} \label{eq:one-dimension-flip-sign-LL}
(s_j - L)(s_{j+1}-L) \leq 0
\end{equation}
and hence
\begin{equation*}
|s_{j+1}-L|\leq |s_{j+1}-s_j|= \left|\sum_{n=jk+1}^{(j+1)k} \eps_n a_n\right|\leq k\cdot |a_{m(j)}|;
\end{equation*}	
\item $s_j - L$ and $s_{j+1} - L$ have the same sign and we have:
\begin{equation} \label{eq:one-dimension-notflip-sign-LL}
|s_{j+1}-L| \leq |s_j-L|-|a_{m(j)}|\leq |s_j-L|.
\end{equation}
\end{enumerate}
Since the sequence $(a_{m(j)})$ is not in $\ell^1$, \eqref{eq:one-dimension-notflip-sign-LL} cannot be true for eventually all $j$; hence infinitely many $j\in\NN$ satisfy \eqref{eq:one-dimension-flip-sign-LL}. For these indices $(j_n)$ it follows that $s_{j_n}\to L$. On the other hand, for other index $j$ we have that $s_{j}$ is bounded by $s_{j_n}$ for some $j_n<j<j_{n+1}$
. This implies that $s_j\to L$ for all $j\ge 0$ and $\sum_{n<N}\eps_n a_n \to L$ as $$|s_{j}- \sum_{n<N}\eps_n a_n|<k\cdot |a_{m(j)}|,$$ where $jk\le N<(j+1)k$ and  $a_{m(j)}\to 0$.

Finally, it is enough to prove that $\dim \Lambda\geq \frac{k-2}{k}$. For all $j\in \NN$ define
\begin{equation*} \Lambda_j=\{\eps\in \Omega \colon (s_i-L)(s_{i+1}-s_i)\geq 0 \text{ and }  |s_{i+1}-s_i|\ge |a_{m(i)}| \text{ for all } 0\leq i< j\}.
\end{equation*}
Clearly, $\Lambda=\cap_{j=1}^{\infty} \Lambda_j$. On the other hand, for every $j\ge 0$ and $\eps\in \Lambda_j$ it follows that
\begin{equation*}
    \#\{\lambda\in \Lambda_{j+1}\colon \lambda \restriction
    (jk)=\eps\restriction  (jk)\}\geq 2^{k-2}.
\end{equation*}

Indeed, for each $\eps\restriction (jk)$ at least a quarter (that is, at least $2^{k-2}$ many) of the possible choices of $(\eps_{kj+1},\ldots,\eps_{k(j+1)})$ satisfy both conditions of $\Lambda_{j+1}$. Indeed, no matter what $(\eps_i: i\in I_j)$ are, consider changing the sign of either $\eps_{m(j)}$ or/and all signs of $(\eps_i: i\in I_j\setminus\{m(j)\})$ simultaneously. For at least one of the four options the signs clearly

satisfy both inequalities defining $\Lambda_{j+1}$.

Thus, for every $\lambda\in \Lambda_j$ and $\tau=\lambda \restriction (jk)$ the set $\{\eps\in \Lambda_{j+1}: \eps\restriction (jk)=\tau\}$ has at least $2^{k-2}$ element. Now we can define a Borel probability measure $\mu$ supported on $\Lambda$ by repeated subdivision as follows. Assume that $\lambda\in \Lambda_j$ and $\tau=\lambda \restriction (jk)$ such that $[\tau]$ has assigned measure $2^{-j(k-2)}$. Choose $2^{k-2}$ clopen sets from the
\begin{equation*} \{[\eps\restriction k(j+1)]: \eps\in \Lambda_{j+1}, \eps\restriction (jk)=\tau\},
\end{equation*}
and assign measure $2^{-(j+1)(k-2)}$ to all of them. Choose $\delta>0$ and $\tau\in \supp \mu$, and suppose that $\delta \in [2^{-(j+1)k},2^{-jk}]$ for some $j\in\NN$. Then
\begin{equation*}
\mu \left(B\big(\tau,\delta \big)\right)\leq \mu \left(B\big(\tau,2^{-jk}\big)\right)=\mu\left([\tau\restriction (jk)]\right)=2^{-j(k-2)}\leq (2^k\delta)^{\frac{k-2}{k}}=C_k \delta^{\frac{k-2}{k}}.
\end{equation*}
The Mass Distribution Principle  (Theorem \ref{thm_MDP}) implies that $\dim(\Lambda) \geq \frac{k-2}{k}$, which completes the proof.
\end{proof}

\section{Convergence in higher dimensions}\label{s3}

The main goal of this section is to prove the following theorem.
\begin{theorem}
\label{thm1hdim}
Let $d\geq 2$. For any sequence $(a_n)$ in $\RR^d$ with $\|a_n\| \to 0$ we have
\begin{equation*}
\dim \left\{ \eps \in \Om: \sum_{n=1}^\infty \eps_n a_n  \text{ is convergent} \right\}=1.
\end{equation*}
\end{theorem}

Even the fact that for any $\|a_n\| \to 0$, there is a sequence of signs $\eps_n$ such that $\sum_{n=1}^\infty \eps_n a_n$ converges is not as immediate as in the real line case. To solve this difficulty, we need the following useful lemma, see \cite[Exercise 2.1.2.]{Ka}, where also its proof is presented.
\begin{lemma}
\label{finsign}
For any finite collection of elements $x_1, x_2, \dots, x_n \in \rr^d$, there is a collection of signs $\eps_1, \eps_2, \dots, \eps_n \in \{-1,1\}$ and an absolute constant $C=C(d)$ such that
\begin{equation*} \max_{1 \leq j \leq n} \left\| \sum\limits_{k=1}^j \eps_k x_k  \right\| \leq C \max_k \|x_k\|.
\end{equation*}
\end{lemma}

As an easy corollary, one can also obtain the following, see \cite[Exercise 1.3.6.]{Ka}. We need the following technical version.

\begin{cor}\label{corka}
Let $(a_n)$ be a sequence in $\RR^d$ with $\|a_n\|\to 0$. Then there is a sequence $\eps\in \Omega$
such that $\sum_{n=1}^\infty \eps_n a_n$ is convergent. Moreover, there is a constant $K=K(d)$ such that
\begin{equation*}
\max_{n} \left\| \sum_{i=1}^n \eps_i a_i \right\|\leq K \max_{n} \left\| a_n \right\|.
\end{equation*}
\end{cor}

\begin{proof}
Let $M=\max_{n} \left\| a_n \right\|$ and for all integers $m \geq 0$ let $q_m$ be the smallest positive integer such that $\|a_n\|\leq M2^{-m}$ for all $n \geq q_m$, and let $Q_m=\{q_m,\dots, q_{m+1}-1\}$.
Using  Lemma \ref{finsign} in each block $Q_m$ separately, we obtain a sequence of signs $\eps_n$ and a constant $C=C(d)$ such that for all $m \geq 1$ we have
\begin{equation*}
\max_{1 \leq j \leq q_{m+1}-1} \left\| \sum\limits_{k=q_m}^j \eps_k a_k  \right\| \leq C \max_{k \in Q_m} \|a_k\| \leq CM2^{-m}.
\end{equation*}
Hence it is easy to see that the sequence $(\sum_{k=1}^n \eps_k a_k)_n$ is Cauchy, and
$\sum_{n=1}^\infty \eps_n a_n$ converges. Moreover, let $K=2C$, then we obtain
\begin{equation*}
\max_{n} \left\| \sum_{i=1}^n \eps_i a_i \right\|\leq 2CM= K\max_{n} \left\| a_n \right\|. \qedhere
\end{equation*}
\end{proof}

Now we are ready to prove Theorem~\ref{thm1hdim}.
\begin{proof}[Proof of Theorem~\ref{thm1hdim}]
Define
\begin{equation*}
\Lambda=\left\{\eps\in \Omega: \sum_{n=1}^{\infty} \eps_na_n \textrm{ converges}\right\}
\end{equation*}
and fix an arbitrary $k\in \NN$. It is enough to prove that
\begin{equation} \label{eq:k-1}
\dim \Lambda \geq \frac{k-1}{k}.
\end{equation}
Our strategy will be to project $\Lambda$ onto a $\frac{k-1}{k}$-dimensional compact metric space $\Theta=\Theta(k)$ defined as follows. For any integer $j\ge 0$ define $I_j$ as
\begin{equation*}
I_j=\{kj+1, \dots, k(j+1)\}.
\end{equation*}
Define the equivalence relation $\sim$ on $\Omega$ such that $\eps\sim \omega$ if and only if
\begin{equation*}
\eps \restriction I_j=\pm \left(\omega \restriction I_j\right) \text{ for all integers } j\geq 0.
\end{equation*}
Let $\Theta=\Omega/\sim$ be the quotient set of $\Omega$ given by the equivalence relation $\sim$, and let $\pi \colon \Omega\to \Theta$ be the natural projection. For $\tau,\rho\in \{-1,1\}^k$ we write $\tau \equiv \rho$ if $\tau=\pm \rho$. Clearly, $\equiv$ is an equivalence relation and let $\mathcal{T}=\{-1,1\}^k/\equiv$ be the quotient set. Then we have
\begin{equation*}
\Theta=\mathcal{T}^{\NN},
\end{equation*}
and define the metric $\varrho$ on $\Theta$ such that if $\theta,\theta'\in \Theta$ with $\theta=(\tau_1,\tau_2,\dots)$ and $\theta'=(\tau'_1,\tau'_2,\dots)$ then
\begin{equation*}
\varrho(\theta,\theta')=2^{-k\min\left\{j: \tau_j\neq \tau'_j\right\}}.
\end{equation*}
Then it is easy to see that Corollary~\ref{corka} implies that $\pi(\Lambda)=\Theta$, and $\pi$ is Lipschitz-$1$. As Lipschitz maps cannot increase the Hausdorff dimension, it is enough to prove that $\dim \Theta\geq \frac{k-1}{k}$. We can define a Borel probability measure $\mu$ on $\Theta$ by repeated subdivision such that for any integer $j\geq 1$ and $\tau_1,\dots, \tau_j \in \mathcal{T}$ for $\chi=\tau_1\dots \tau_j\in \mathcal{T}^{j}$ we have
\begin{equation*}
\mu([\chi])=2^{-j(k-1)};
\end{equation*}
indeed, note that the number of such words $\chi$ is $2^{j(k-1)}$, so $\mu$ is actually the uniform distribution on $\Theta$.

Fix an arbitrary $0<r<2^{-k}$ and let $\theta\in \Theta$. Suppose that $r \in [2^{-(j+1)k},2^{-jk}]$ for some integer $j\geq 1$. Then
\begin{equation*}
\mu \left(B\big(\theta,r \big)\right)\leq \mu \left(B\big(\theta,2^{-jk}\big)\right)=\mu\left([\theta\restriction (jk)]\right)=2^{-j(k-1)}\leq (2^kr)^{\frac{k-1}{k}}=C_k r^{\frac{k-1}{k}},
\end{equation*}
where $C_k=2^{k-1}$. Therefore, the Mass Distribution Principle (Theorem \ref{thm_MDP}) implies that $\dim(\Lambda) \geq \frac{k-1}{k}$. This completes the proof.
\end{proof}

\section{Signed sums in higher dimensions} \label{s:sums}

In this section, we turn our attention to the higher-dimensional setting and examine the extent to which the phenomena observed in one dimension persist in $\mathbb{R}^d$ for $d \geq 2$. Although Theorem~\ref{thm1dim} provides a complete characterization in the one-dimensional case, its direct generalization fails in higher dimensions. This motivates a deeper analysis of the summability behavior of null sequences and their associated sum sets in $\mathbb{R}^d$.

A natural starting point is the following fundamental question.

\begin{question}
What are the null sequences $(a_n)$ in $\RR^d$ (that is, $\|a_n\| \to 0$) such that $S(a_n)=\RR^d$?
\end{question}

A basic necessary condition for this to hold is that the sequence must remain non-summable in every direction, that is,
\begin{equation}
\sum_{n=1}^\infty |\langle a_n, e \rangle | = \infty \quad \text{for all } e \in \mathbb{S}^{d-1},
\tag{LS}\label{LS}
\end{equation}
where $\langle\cdot,\cdot\rangle$ denotes the standard scalar product on $\RR^d$. In other words, the sum of the length of any projection of $(a_n)$ is not in $\ell_1$, otherwise we clearly cannot reach all elements of $\RR^d$. This condition is also central in the classical L\'evy-Steinitz theorem, which characterizes the set of sums of rearranged series in finite-dimensional spaces:

\begin{theorem}[L\'evy-Steinitz theorem]\label{thmLS}
Let $\sum_{n=1}^\infty  a_n$ be a conditionally convergent series in $\RR^d$. Then $(a_n)$ satisfies \eqref{LS} if and only if
\[
\left\{\sum_{n=1}^\infty  a_{\pi(n)} \mid \pi\colon  \NN \to \NN \text{ is a permutation}\right\} = \RR^d.
\]
\end{theorem}

While it may be tempting to conjecture that the divergence condition \eqref{LS} also guarantees that the set of all subseries sums coincides with $\RR^d$, this turns out to be false. As shown in \cite{BGM}, a more nuanced picture emerges when considering the \emph{achievement set} of a conditionally convergent series, defined as
\[
A(a_n) = \left\{\sum_{n=1}^\infty \varepsilon_n a_n : \varepsilon_n \in \{0,1\}^\NN \right\}.
\]
Although condition \eqref{LS} implies that $A(a_n)$ is dense in $\RR^d$ (see \cite[Lemma 3.2]{BGM}), it is insufficient to ensure that $A(a_n) = \RR^d$. The remainder of this section explores this discrepancy and presents partial positive results in this higher-dimensional context.

\begin{prop}\label{propcex1}
If \eqref{LS} holds, then $S(a_n)$ is dense in $\RR^d$.
\end{prop}

\begin{proof}
Suppose that $(a_n)$ satisfies \eqref{LS}.
Let $x\in \R^d$ and $\delta>0$ be fixed, we show that there exists an $\eps\in \Om$ such that $\| x-\sum_{n=1}^{\infty}\eps_n a_n\|<\delta$. Let $N>0$ such that $\max_n \|a_n\|<\frac{\delta}{2K}$ for every $n\geq N$, where $K=K(d)$ is the absolute constant given by Corollary \ref{corka}.
Let $s=\sum_{i=1}^{N-1} a_i$. Now it is enough to show that there are signs
$(\eps_n)_{n\geq N}$ such that
\begin{equation}
\label{eq:bigN}
\left\|(x-s)-\sum_{n=N}^{\infty}\eps_n a_n\right\|<\delta.
\end{equation}
Indeed, \eqref{eq:bigN} yields that $\| x-\sum_{n=1}^{\infty}\eps_n a_n\|<\delta$ if $\eps_i=1$ for all $i\in \{1,\dots,N-1\}$.

First note that if $(a_n)_{n\geq 1}$ satisfies \eqref{LS}, then $(a_n)_{n\geq N}$  also satisfies it. Clearly, every sequence that satisfies \eqref{LS} is automatically conditionally convergent. Let $\NN_{\geq N}=[N,\infty)\cap \NN$, by Theorem \ref{thmLS} we obtain
\begin{equation*} \left\{\sum_{n=N}^\infty  a_{\pi(n)}: \pi\colon \NN_{\geq N} \to \NN_{\geq N} \text{ is a permutation}\right\}=\RR^d.
\end{equation*}
Thus there is a permutation $\pi\colon  \NN_{\geq N} \to \NN_{\geq N}$ such that $\sum_{n=N}^\infty  a_{\pi(n)}=x-s$. Hence there is a set $H\subset \NN_{\geq N}$ such that
\begin{equation} \label{eq:extraN}
\left\|(x-s)-\sum_{n\in H} a_n\right\|<\frac{\delta}{2}.
\end{equation}
Let $H'=\NN_{\geq N}\setminus H$ and we apply Corollary \ref{corka} to $(a_n)_{n\in H'}$.
Since $\|a_n\|<\frac{\delta}{2K}$ for $n\geq N$, we can find signs $\eps_{n}$ for $n\in H'$ such that
\begin{equation} \label{eq:H'}
\left\|\sum_{n\in H'} \eps_n a_n\right|\leq K\cdot \frac{\delta}{2K}=\frac{\delta}{2}.
\end{equation}
Let $\eps_n=1$ if $n\in H$, so $\eps_n$ is defined for all $n\geq N$. Then \eqref{eq:extraN} and \eqref{eq:H'} imply
\begin{equation*}
\left\|(x-s)-\sum_{n=N}^{\infty} \eps_n a_n\right\|\leq \left\|(x-s)-\sum_{n\in H} a_n\right\|+\left\|\sum_{n\in H'}\eps_n a_n\right\|<\frac{\delta}{2}+\frac{\delta}{2}=\delta.
\end{equation*}
Thus \eqref{eq:bigN} holds, and the proof is complete.
\end{proof}

On the other hand, adopting the construction in \cite[Example~3.13.]{BGM} we have the following.
\begin{example}\label{propcex2}
There is a null sequence $(a_n)$ in $\RR^2$ satisfying \eqref{LS} and $S(a_n) \neq \RR^2$.
\end{example}

\begin{proof}
Let $n_0=0$ and define the increasing sequence $n_{k+1}=n_k+2^{10^{k^{2}}+1}$.
Let $a_n=\left(\frac{(-1)^n}{2^{10^{k^2}}},\frac{(-1)^n}{2^k}\right)$, where $k$ is the only positive integer for which $n\in (n_{k-1}, n_k]$. A straightforward adaptation of the proof \cite[Example~3.13.]{BGM} yields that \eqref{LS} holds for $(a_n)$ and $S(a_n)\subset L \times \RR$, where $L$ denotes the set of Liouville numbers.
\end{proof}

The following definition turns out to be useful.

\begin{defi}
Let $(a_n)$ be a sequence in $\RR^d$ tending to $0$.  
We say that $u \in \mathbb{S}^{d-1}$ is a \emph{L\'evy vector} of $(a_n)$ if for every $\varepsilon>0$ we have
\begin{equation*}
   \sum_{n=1}^{\infty} \|a_n\|\,
   \chi\!\left(\frac{a_n}{\|a_n\|} \in B(u,\varepsilon)\right)
   = \infty,
\end{equation*}
where $\chi(E)$ denotes the indicator function of a condition $E$ (i.e., $\chi(E)=1$ if $E$ holds and $\chi(E)=0$ otherwise).
\end{defi}

The following fact has been claimed by Halperin \cite[$\S$~3.3]{Ha} without proof, see also \cite[Proposition 2.2]{GM}. For the readers' convenience, we include an easy proof.
\begin{prop}
\label{levyfact}
Let $(a_n)$ be a null sequence in $\RR^d$, then the following are equivalent:
\begin{enumerate}[(i)]
\item \label{i:Levy1} $u$ is a L\'evy vector of the sequence $(a_n)$;
\item \label{i:Levy2} there exist an index set $I\subset \NN$ sequences $(\alpha_{n})$ in $\RR^{+}$ and $(\omega_{n})$ in $\RR^d$ such that for all $n\in I$ we have
\begin{equation*}
a_{n}=\al_{n} u + \omega_{n} \text{ with }  \sum_{n\in I} \al_{n} = \infty   \text{ and } \sum_{n\in I} \|\om_{n}\| < \infty.
\end{equation*}
\end{enumerate}
\end{prop}
\begin{proof}
First we show that \eqref{i:Levy2} implies \eqref{i:Levy1}. Let $\gamma>0$ be arbitrarily given. For all $n\in I$ with $\delta:=\gamma/2$ we obtain by an easy calculation that
\begin{equation} \label{eq:norm1}
 \|\omega_{n}\|\leq \delta\alpha_{n}  \quad \text{implies} \quad \frac{a_n}{\|a_n\|}\in B(u,\gamma).
\end{equation}
Since
\begin{equation*}
\sum_{n\in I} \alpha_{n} \chi(\delta\alpha_{n}  \leq\|\omega_{n}\|)\leq \frac{1}{\delta} \sum_{n\in I} \| \omega_{n}\|<\infty,
\end{equation*}
our assumption ${\displaystyle \sum_{n\in I} \al_{n} = \infty}$ yields that
\begin{equation} \label{eq:norm2}
\sum_{n\in I} \alpha_{n} \chi(\delta \alpha_{n} >\|\omega_{n}\|)=\infty.
\end{equation}
Then \eqref{eq:norm1} and \eqref{eq:norm2} imply that
\begin{equation*}
\sum_{n=1}^{\infty} \|a_{n}\|\chi\left(\frac{a_{n}}{\|a_{n}\|} \in B(u,\gamma)\right)\geq \sum_{n\in I} \alpha_{n} \chi(\delta \alpha_{n}  >\|\omega_{n}\|)=\infty,
\end{equation*}
which implies that $u$ is a L\'evy vector of the sequence $(a_n)$.

Now we show that \eqref{i:Levy1} yields \eqref{i:Levy2}. We may assume without loss of generality that $\frac{a_n}{\|a_n\|}\in B(u,1)$ for all $n$. Therefore, we can uniquely write
\begin{equation*} a_n=\alpha_n u+\omega_n \text{ such that } \alpha_n>0 \text{ and } \omega_n \perp u \text{ for all } n\geq 1.
\end{equation*}
We may also assume $\|a_n\|\leq 1$ for all $n$.
Let $\gamma\leq 1/2$ and $n$ be arbitrary, we show that
\begin{equation} \label{eq:norm3}
\frac{a_n}{\|a_n\|}\in B(u,\gamma)  \quad \text{implies} \quad  \|\omega_{n}\|\leq 2\gamma \alpha_{n}.
\end{equation}
Indeed, our assumption yields that $\|\omega_n\|\leq \gamma \|a_n\|$, so
$\|\omega_n\|^2\leq \gamma^2 (\alpha_n^2+\|\omega_n\|^2)$, which implies that
\begin{equation*}
\|\omega_n\|\leq \frac{\gamma}{\sqrt{1-\gamma^2}}\alpha_n\leq 2\gamma \alpha_n,
\end{equation*}
so \eqref{eq:norm3} holds. We will define subsets $I_0,I_1,\dots \subset \NN$ such that
\begin{enumerate}[(1)]
\item \label{ii1} $I_0=\emptyset$,
\item \label{ii2} $\max I_{m-1}<\min I_{m}$ for all $m\geq 1$,
\item \label{ii3} $1\leq \sum_{n\in I_m} \alpha_n \leq 2$ for all $m\geq 1$,
\item \label{ii4} $\|\omega_n\|\leq 2^{1-m}\alpha_n$ for all $n\in I_m$ and $m\geq 1$.
\end{enumerate}
Let $I_0=\emptyset$, so \eqref{ii1} is satisfied. Assume by induction that $I_0,\dots,I_{m-1}$ are already defined. Choose $\gamma_{m}=2^{-m}$. As $u$ is a Levy vector of the sequence $(a_n)$, by definition
\begin{equation*} \sum_{n=1}^{\infty} \|a_n\| \chi\left(\frac{a_n}{\|a_n\|}\in B(u,\gamma_{m})\right)=\infty.
\end{equation*}
Let $I_{m}$ be a minimal subset of $\left\{n: \frac{a_n}{\|a_n\|}\in B(u,\gamma_{m})\right\}\setminus \{1,\dots,\max I_{m-1}\}$ such that $\sum_{n\in I_m} \alpha_n\geq 1$, then clearly \eqref{ii2} holds. The minimality of $I_m$ and $\|a_n\|\leq 1$ imply that $\sum_{n\in I_m} \alpha_n\leq 2$, so \eqref{ii3} holds as well. Then \eqref{eq:norm3} yields that for all $n\in I_m$ we have $\|\omega_n\|\leq 2\gamma_m \alpha_n=2^{1-m}\alpha_n$.

Finally, let $I=\cup_{m=1}^{\infty} I_m$. Then \eqref{ii3} implies that $\sum_{n\in I} \alpha_n\geq \sum_{m=1}^{\infty} 1=\infty$. Using \eqref{ii4} and \eqref{ii3} we obtain that
\begin{equation*}
\sum_{n\in I} \|\omega_n\|\leq \sum_{m=1}^{\infty} \sum_{n\in I_m} 2^{1-m} \alpha_n\leq  \sum_{m=1}^{\infty} 2^{2-m}<\infty.
\end{equation*}
Thus \eqref{i:Levy2} holds, and the proof is complete.
\end{proof}

We prove the following theorem, giving a sufficient condition for $S(a_n) = \rr^d$. For an analogous result for achievement sets in the plane, see \cite[Theorem 2.6]{GM}. First, we need the following straightforward lemma.

\begin{lemma}
\label{lem1dim}
Let $(\al_n)$ be a null sequence in $\RR$ with $\sum_{n=1}^\infty |\al_{n}| = \infty$. Let $x \in \RR$, $\delta>0$, and $N \in \NN$.
Then there exist $K \geq N$ and a sequence of signs $(\eps_n)$ such that for all $k\in \N$
\begin{equation}\label{eqrei1}
|x-\sum_{n=1}^{k} \eps_n \al_n| < |x|+ \max_{n\in \N} |\al_n|,
\end{equation}

 and for any $k \geq K$ we have
\begin{equation}\label{eqrei2}
x-\delta < \sum_{n=1}^{k} \eps_n \al_n < x+\delta.
\end{equation}
\end{lemma}

\begin{proof}
We construct the sequence inductively, similar in spirit to Riemann's Rearrangement Theorem.
We may assume that $\eps_k$ is already defined for all $k<n\in \N$. Then for $n\in \N$ we define
$$\eps_n=\begin{cases}
 1   & \textrm{ if } \sum_{k=1}^{n-1} \eps_k \al_k\le  x,
  \\
 -1 & \textrm{ if } \sum_{k=1}^{n-1} \eps_k \al_k> x.
\end{cases}$$
In this case, it is clear that \eqref{eqrei1} holds for all $n\in \N$. Since $\al_n\to 0$ and $\sum_{n=1}^\infty |\al_{n}| = \infty$, for every $N\in \NN$ and $\delta>0$ there exists $K\ge N\in \NN$ such that $\al_k\le \delta/2$ and $|x-\sum_{n=1}^{k} \eps_n \al_n|<\delta/2$ for some $k\ge K$, and hence \eqref{eqrei1} implies \eqref{eqrei2}.
\end{proof}

\begin{theorem}
\label{thmlevy}
Let $(a_n)$ be a null sequence in $\R^d$ with $d$ linearly independent L\'evy vectors.
Then $S(a_n) = \R^d$.
\end{theorem}

\begin{proof}
Without loss of generality, we may assume that $\norm{a_n}<1$ for all $n\in \N$ by omitting finitely many initial terms of the sequence if necessary. This ensures that the infinite sum $\sum \varepsilon_n a_n$ converges whenever $\varepsilon_n \in \{-1,1\}$ and also allows term-by-term addition with arbitrary signs.

Let $x \in \R^d$ be arbitrary. We aim to construct a sign sequence $\eps \in \Omega$ such that $\sum_{n=1}^\infty \eps_n a_n = x$. Let $u_1, u_2, \dots, u_d$ be $d$ linearly independent L\'evy vectors of the sequence $(a_n)$ with $\norm{u_i}=1$ for each $i$.

By Proposition~\ref{levyfact}, for each $i=1,\dots,d$ there exists a subsequence $(k_n^i)_n$ and positive reals $\alpha_{k_n^i}$ such that
\begin{equation}
\label{eq:levy_decomp}
a_{k_n^i} = \alpha_{k_n^i} u_i + \omega_{k_n^i}, \quad \text{where }
\alpha_{k_n^i} \to 0,\quad \sum_{n=1}^\infty \alpha_{k_n^i} = \infty,\quad \text{and } \sum_{n=1}^\infty \norm{\omega_{k_n^i}} < \infty.
\end{equation}
We may assume the subsequences $(k_n^i)$ are pairwise disjoint. Let $(k_n^0)$ denote the increasing enumeration of the remaining indices, i.e., $\N \setminus \bigcup_{i=1}^d \{k_n^i : n \in \N\}$.

Write $x = \beta_1 u_1 + \dots + \beta_d u_d$ for some real coefficients $\beta_1, \dots, \beta_d$.

We now construct recursively a sign sequence $\eps \in \Omega$ and a strictly increasing sequence of integers $1 = q_0 < q_1 < q_2 < \dots$ such that
\begin{equation}
\label{eq_qm}
\norm{ \sum_{j=1}^T \varepsilon_j a_j - x } \le \frac{C}{2^m} \quad \text{for all } T \in [q_m, q_{m+1}),
\end{equation}
for some absolute constant
\[
C = \begin{cases}
\sum_{i=1}^d (1 + K_i) + K_0 & \text{if } m = 1, \\
d + K_0 & \text{if } m > 1,
\end{cases}
\]
where $K_i = \sum_{n=1}^\infty \norm{\omega_{k_n^i}}$ for $i = 1,\dots,d$, and $K_0 = K_0(d, a_n)$ is the constant given by Corollary~\ref{corka}.

We also define a sequence $(y_m)_{m\ge 0} \subset \R^d$ such that $\sum_{t=0}^{m-1} y_t \to x$ uniformly, where
\[
y_t = \sum_{j = q_{t-1}+1}^{q_t} \varepsilon_j a_j
\]
for a suitable choice of signs $\varepsilon_j$.

The construction proceeds recursively. Suppose that $q_t$ and $y_t$ are already defined for $t = 0,\dots,m-1$, and set
\[
z_{m-1} := x - \sum_{t=0}^{m-1} y_t,
\]
with $\norm{z_{m-1}} < \frac{C}{2^{m-1}}$. Write
\[
z_{m-1} = \beta_1^{(m-1)} u_1 + \dots + \beta_d^{(m-1)} u_d
\]
for some real numbers $\beta_i^{(m-1)}$.

Now choose $q_m > q_{m-1}$ large enough so that:
\begin{enumerate}
\item $\norm{a_n} < 1/2^m$ for all $n \geq q_m$, and $\sum_{k_n^i \geq q_m} \norm{\omega_{k_n^i}} < 1/2^m$ for all $i=1,\dots,d$.
\item There exist sign sequences $(\eps_{k_n^i})_{n}$ such that for all $k \geq q_m$, as far as needed,
\[
\beta_i^{(m-1)} - \frac{1}{2^m} < \sum_{k_n^i \in [1,k]} \eps_{k_n^i} \alpha_{k_n^i} < \beta_i^{(m-1)} + \frac{1}{2^m}, \quad \text{for all } i = 1,\dots,d.
\]
\end{enumerate}
Such a $q_m$ exists by Lemma~\ref{lem1dim}.

Define the index sets
\[
Q_m = \N \cap [q_{m-1}, q_m), \qquad F_m = \{k_n^0\}_n \cap [q_{m-1}, q_m).
\]
By Lemma~\ref{finsign}, there exists a sign sequence $(\eps_l)_{l \in F_m}$ such that
\[
\max_{q_{m-1} \le j \le \max F_m} \left\| \sum_{l \in F_m \cap [q_m, j]} \eps_l a_l \right\| \le \frac{K_0}{2^{m-1}},
\]
using the fact that $\norm{a_n} < 1/2^{m-1}$ for all $n > q_{m-1}$.

Now define
\[
y_m = \sum_{j=q_{m-1}+1}^{q_m} \eps_j a_j = \sum_{i=1}^d \sum_{k_n^i \in Q_m} \eps_{k_n^i}(\alpha_{k_n^i} u_i + \omega_{k_n^i}) + \sum_{k_n^0 \in Q_m} \eps_{k_n^0} a_{k_n^0}.
\]
Then:
\begin{align*}
&\left\| y_m + \sum_{t=0}^{m-1} y_t - x \right\| = \norm{y_m - z_{m-1}} \le \\
&\sum_{i=1}^d \left\| \sum_{k_n^i \in Q_m} (\eps_{k_n^i} \alpha_{k_n^i} - \beta_i^{(m-1)}) u_i \right\|
+ \sum_{i=1}^d \sum_{k_n^i \in Q_m} \norm{\omega_{k_n^i}}
+ \left\| \sum_{k_n^0 \in Q_m} \eps_{k_n^0} a_{k_n^0} \right\| \le \\
&\begin{cases}
\sum_{i=1}^d (1 + K_i) + K_0 & \text{if } m = 1, \\
d \cdot \left(\frac{1}{2^m} + \frac{1}{2^m}\right) + \frac{K_0}{2^{m-1}} & \text{if } m > 1.
\end{cases}
\end{align*}
This shows that the partial sums converge to $x$ with exponentially decreasing error. Moreover, applying \eqref{eqrei1} in Lemma \ref{lem1dim} we obtain that for all $l \ge q_m$, the estimate
\[
\norm{ \sum_{j=1}^l \varepsilon_j a_j - x } \le \frac{d}{2^{m-1}} + \frac{C}{2^m}
\]
holds. Therefore, the whole signed series converges to $x$.
\end{proof}

Now we can formulate a possible analogue of the second part of Theorem \ref{thm1dim}.

\begin{theorem}\label{thmdhdim}
Let $(a_n)$ be a sequence in $\RR^d$ with $d$ linearly independent Lévy vectors and $\|a_n\|\to 0$, then for all $L \in \RR^d$ we obtain
\begin{equation}\label{eq:l2}	
\dim\left\{\eps \in \Om: \sum_{n=1}^\infty  \varepsilon_n a_n =L \right\}=1.
\end{equation}
\end{theorem}

\begin{proof} The proof is a combination of the proofs of Theorem~\ref{thm1hdim} and Theorem~\ref{thmlevy}. 
For any $L\in \R^d$ we define
\begin{equation*}
\Lambda_L=\left\{\eps\in \Omega: \sum_{n=1}^{\infty} \eps_na_n=L\right\}
\end{equation*}
and fix an arbitrary $k\ge d+2\in \NN$. We prove that
\begin{equation} \label{eq2:k-1}
\dim \Lambda_L \geq \frac{k-d-1}{k}.
\end{equation}
As $d$ is fixed, if $k\to \infty$ we get that $\dim\Lambda_L=1$.
For any integer $j\ge 0$ we define $I_j$ as
$$I_j=\{jk+1, \dots k(j+1)\}.$$
Let $u_1, \dots, u_d$ denote the linearly independent Lévy vectors of the sequence $(a_n)$ in $\R^d$. Thus there are subsequences $a_{k_n^i}$ satisfying \begin{equation*}
 a_{k_n^i}=\al_{k_n^i} u_i + \om_{k_n^i}, \text{ where }
\al_{k_n^i} \to 0,~\sum_{n=1}^\infty \al_{k_n^i} = \infty, \ \text{and} \ \sum_{n=1}^\infty \|\om_{k_n^i}\| < \infty.
\end{equation*}
Without loss of generality we can assume that the subsequences $(k_n^i)$ are pairwise disjoint and each of them contains at most one element of each $I_j$. Let us denote $\cup_{i=1}^d (k_n^i)$ by  $G$. If for a given $i$ and $j$ the sequence $(k_n^i)$ does not contain any element of $I_j$, then we choose an arbitrary element of $I_j$ additionally to guarantee that for each $j\in \N$ we select exactly $d$ elements from $I_j$. The selection of these extra elements are denoted by $H$. Let $\widetilde{I}_j$ denote the $I_j\setminus (\cup_{i=1}^d (k_n^i)\cup H)$. Hence $\widetilde{I}_j$ is of cardinality exactly $k-d$. Let $\widetilde{\Omega}=\Omega\restriction \cup_{j=1}^{\infty} \widetilde{I}_j$.
\begin{itemize}
\item
For the $\widetilde{\Omega}$ we proceed the same way as in Theorem \ref{thm1hdim} to prove that $$\dim\{\eps\in \Omega\colon \sum_{n\in \cup_j\widetilde{I}_j} \eps_na_n \textrm{ is convergent} \}\ge \frac{k-d-1}{k}.$$
\item
For the elements of $H$ we can apply Corollary \ref{corka},  which implies that there exists a sign such that $\sum_{n\in H} \eps_na_n$ is convergent.
\item Finally, for $G$ we apply similar argument as in Theorem \ref{thmlevy} to guarantee that the whole signed sequence is converging to $L$. To this let us fix a sign sequence $(\eps_n)_{n\in \cup_j \widetilde{I}_j}$ so that $\sum_{n\in \cup\widetilde{I}_j} \eps_na_n=L_1$. We also fix a sign sequence $(\eps_n)_{n\in H}$ such that $\sum_{n\in H}\eps_na_n=L_2$. Since $(a_n)_{n\in G}$ has n linearly independent Levy vectors (namely $u_i$ ($i=1,\dots, d)$), by Theorem \ref{thmlevy} we have that $S(a_{n})_{n\in G}=\R^d$. Thus, there is a sign sequence $(\eps_n)_{n\in G}$ such that $\sum_{n\in G}\eps_na_n=L-L_1-L_2$.
\end{itemize}
Hence $$\sum_{n=1}^{\infty} \eps_na_n=\sum_{n\in \cup_j\widetilde{I_j}}\eps_na_n+\sum_{n\in H}\eps_na_n+\sum_{n\in G}\eps_na_n=L_1+L_2+(L-L_1-L_2)=L.$$

Thus, for any fixed choice of signs in $\widetilde{\Omega}$, whenever the corresponding partial sum converges, we can choose the remaining signs in $\Omega \setminus \widetilde{\Omega}$ so that the entire series converges to $L$.
This shows that $\dim(\Lambda_L) \ge \tfrac{k-d-1}{k}$ holds for each $k \in \mathbb{N}$.
Letting $k \to \infty$, we obtain $1 \le \dim(\Lambda_L) \le \dim(\Omega) = 1$, which completes the proof.
\end{proof}

\section{Open questions}\label{sec:open}
In the one-dimensional case we proved that
\begin{equation*} 
\dim\left\{\varepsilon \in \Omega : \sum_{n=1}^\infty \varepsilon_n a_n = L \right\} = 1
\end{equation*}
for all $L \in \RR$, provided $(a_n) \notin \ell^{1}$. In higher dimensions we do not even know for which $L \in \RR^d$ the corresponding set has Hausdorff dimension $1$. We were only able to prove that the analogue result holds for all $L \in \RR^d$ if the sequence $(a_n)$ has $d$ linearly independent Lévy vectors. One of the difficulties is that one can construct a sequence $(a_n)$ for which there are some $L \in \RR^d$ such that
\begin{equation*}
\sum_{n=1}^{\infty} \varepsilon_n a_n \not\to L \qquad \text{for any } \varepsilon \in \Omega.
\end{equation*} 
This, however, leads to the following question.

\begin{question}\label{q1}
Let $(a_n)$ be a null sequence such that for every $L \in \RR^d$ there exists $\varepsilon \in \Omega$ with $\sum_{n=1}^\infty \varepsilon_n a_n = L$. Is it true that for all $L\in \RR^d$ we have
\begin{equation*} 
\dim\left\{\varepsilon \in \Omega : \sum_{n=1}^\infty \varepsilon_n a_n = L \right\} = 1 ?
\end{equation*}
\end{question}

We showed that the answer is affirmative if $(a_n)$ has $d$ linearly independent Lévy vectors. On the other hand, Marchwicki and Vlas\'ak \cite{filo} proved the following:

For any $d \in \NN$, $\alpha_1, \dots, \alpha_d \in (0,1]$ and $c_1, \dots, c_d \in \RR \setminus \{0\}$, one has
\[
A\left(c_1\frac{(-1)^{n}}{n^{\alpha_1}}, \dots, c_d\frac{(-1)^{n}}{n^{\alpha_d}} \right) = \RR^d
\]
if and only if the numbers $\alpha_1, \dots, \alpha_d$ are pairwise distinct.

This suggests the following analogue question, which we state in its simplest form for $\RR^2$.

\begin{question}
Is it true that for any distinct $\alpha_1,\alpha_2 \in (0,1]$ we have 
\begin{equation*} 
S\left(\frac{1}{n^{\alpha_1}}, \frac{1}{n^{\alpha_2}}\right) = \RR^2?
\end{equation*} 
\end{question}

If the answer is affirmative, then we may ask the following, which can be seen as a special case of Question~\ref{q1}.

\begin{question}
Let $\alpha_1,\alpha_2 \in (0,1]$ be distinct, and let $a_n=({n^{-\alpha_1}},n^{-\alpha_2}) \in \RR^2$. Suppose that $S(a_n) = \RR^2$. Is it true that, for any $L \in \RR^2$, we have 
\begin{equation*}
\dim\left\{\varepsilon \in \Omega : \sum_{n=1}^\infty \varepsilon_n a_n = L \right\} = 1 ?
\end{equation*}
\end{question}

\subsection*{Acknowledgments} We are indebted to P.~Mattila for the problem and to B.~Solomyak for communicating it to us. We are grateful to B.~Sewell for some helpful suggestions.

\end{document}